\documentclass[11pt,reqno]{amsart}
\usepackage[utf8]{inputenc}
\usepackage[T1]{fontenc}
\usepackage[english]{babel}
\usepackage{amsmath,amssymb,amsthm,mathtools}
\usepackage{graphicx}
\usepackage{booktabs}
\usepackage[margin=2.6cm]{geometry}
\usepackage{xcolor}
\usepackage[colorlinks=true,linkcolor=blue!45!black,citecolor=green!40!black,
            urlcolor=blue!45!black]{hyperref}

\newcommand{\Q}{\mathbb{Q}}
\newcommand{\R}{\mathbb{R}}
\newcommand{\C}{\mathbb{C}}
\newcommand{\rank}{\operatorname{rank}}
\newcommand{\norm}[1]{\lVert #1\rVert}
\newcommand{\Gal}{\operatorname{Gal}}

\theoremstyle{plain}
\newtheorem{theorem}{Theorem}[section]
\newtheorem{proposition}[theorem]{Proposition}

\theoremstyle{definition}
\newtheorem{definition}[theorem]{Definition}
\newtheorem{remark}[theorem]{Remark}

\begin{document}

\title[Origametry and antagonistic heptagon graphs]
{Origametry and antagonistic heptagon graphs: the braced $\{35,1\}$ and the
complete $\{21,2\}$}

\author{Haroldo Costa Silva Filho}
\address{Professor, CEFET-RJ; and PhD student, PGCComp --- Graduate Program in
Computational Sciences, UERJ, Rio de Janeiro, Brazil}
\email{haroldo.filho@cefet-rj.br}

\author{Edward Pegg Jr.}
\address{Wolfram Research, Champaign, Illinois, USA}
\email{edp@wolfram.com}

\subjclass[2020]{52C25, 05C62, 12F10, 51M15, 68W30}
\keywords{Unit-distance graph, rigidity, Laman number, Galois theory, origami,
Huzita--Hatori folds, mountain--valley assignment, regular heptagon, straightedge
and compass}

\begin{abstract}
A folded sheet of paper can construct lengths that a straightedge and compass never
can --- among them $\cos\tfrac{2\pi}{7}$, the proportion hidden in a regular
heptagon. Fold as cleverly as you like, though, and some lengths still slip away.
This paper tells that story through two small graphs shaped like heptagons that turn
out to be complete opposites. Both are \emph{unit-distance graphs} --- drawings in
the plane with every edge exactly one unit long --- and the plainest heptagon of all
is the $7$-cycle, seven unit rods hinged in a ring that settles into a regular
heptagon but flexes like a slack chain. Brace that ring rigid and you get the
\emph{complete} heptagon $\{21,2\}$, a single fixed shape whose corners land on
$\cos\tfrac{2\pi}{7}$; one Beloch fold reaches it, because folding solves the cubic
$8x^3+4x^2-4x-1$. Delete two of its vertices and you get the \emph{braced} heptagon
$\{35,1\}$: still only just rigid, yet with thousands of possible shapes --- its
generic realization count is $2^{6}\cdot103\cdot587$, so it looks as though its
corners should escape folding entirely. That is the surprise the paper resolves:
the regular braced heptagon is still reachable by paper --- one cubic Beloch fold
followed by a single quadratic --- because the angle that pins it down, a root of an
irreducible degree-$12$ palindrome, collapses to a quadratic over the heptagon's own
cubic field $\Q(\cos\tfrac{2\pi}{7})$. Both heptagons therefore sit \emph{inside}
Alperin's field of \emph{origami numbers}; the primes $103,587$ belong to the generic
count, not to the folded shape. Between them runs the whole tale: rigidity (Laman's
theorem and the count of realizations), Galois theory, and origametry. The fold
choices of the construction turn out to be mountain--valley assignments, and every
claim is checked in exact arithmetic.
\end{abstract}

\maketitle

\section{Introduction}

Which lengths can you actually construct? A straightedge and compass give one answer
--- Gauss's --- while a folded sheet of paper gives a larger one, reaching
$\cos\tfrac{2\pi}{7}$, the proportion of a regular heptagon that the compass cannot.
This paper follows that thread to two graphs sitting on opposite sides of what
folding can do; but they are easiest to meet from the ground floor.

The simplest unit-distance graphs are the ones we can already picture. The cycle
$C_n$ is realized by the vertices of a regular $n$-gon, and the Petersen graph by
a symmetric ten-point drawing. Others sit at the heart of a famous open problem:
the Moser spindle --- seven vertices hinged from two rhombi --- and the ten-vertex
Golomb graph are unit-distance graphs that require four colors, and in 2018 de Grey
exhibited one requiring five, lifting the chromatic number of the plane (the
Hadwiger--Nelson problem) into $\{5,6,7\}$ after decades pinned at $\{4,\dots,7\}$; a non-constructive proof
that $\chi(\R^2)\ge6$ has since been announced~\cite{OpenAI6}.
What makes the family subtle is that not every graph fits: the complete graph $K_4$
has \emph{no} unit-distance realization in the plane at all, since four
pairwise-equidistant points require a third dimension, where they occupy the
vertices of a regular tetrahedron. Between these extremes lies a rich middle ground
in which a graph's combinatorics, its rigidity, and the arithmetic of the
coordinates that realize it are bound together --- the binding that separates a
graph a straightedge and compass can draw from one that only paper folding reaches.
This paper studies that binding for two graphs shaped like heptagons, twins in
appearance but antagonists in rigidity and in coordinate field: the braced heptagon
$\{35,1\}$ and the complete heptagon $\{21,2\}$.

Formally, a \emph{unit-distance graph} (UDG) in the plane is a graph $G=(V,E)$
admitting a map $p:V\to\R^2$ with $\norm{p_u-p_v}=1$ for every edge
$\{u,v\}\in E$; the
realization is \emph{faithful} if in addition $\norm{p_u-p_v}\neq1$ for every
non-adjacent pair. Recognizing UDGs is $\exists\R$-complete~\cite{Schaefer}, which
motivates the study of structured families. We follow Pegg's Wolfram \textsc{GraphData} identifiers, in which the pair
$\{a,b\}$ is the database index of the graph, \emph{not} a vertex or edge count:
the \emph{braced} heptagon $\{35,1\}$ ($19$ vertices, $35$ edges) and the
\emph{complete} heptagon $\{21,2\}$ ($21$ vertices, $42$ edges), the latter being
$\mathtt{GraphData[\{"UnitDistance",\{21,2\}\}]}$. \emph{Neither} heptagon is
straightedge-and-compass constructible --- already the regular $7$-gon is not, by
Gauss--Wantzel, since $7$ is not a Fermat prime --- so the natural setting is
\emph{origametry}~\cite{Hull}: the construction of each realization by Huzita--Hatori
folds, and the distinction between the quadratic fold $O5$ (straightedge-and-compass
strength) and the cubic Beloch fold $O6$. The two graphs are \emph{antagonists}: the
complete $\{21,2\}$ is over-braced and globally rigid, with coordinates in a degree-$12$ abelian field over the cyclic cubic
$\Q(\cos\tfrac{2\pi}{7})$, while
the braced $\{35,1\}$ --- a two-vertex deletion of it --- is isostatic with thousands
of realizations; yet, as we prove (Theorem~\ref{thm:origami35}), its unit realization
is again \emph{origami-constructible}, a cubic (trisection) step followed by
quadratics, so the two heptagons are antagonists in rigidity but kin in
constructibility.

\section{Definitions}
\label{sec:defs}

\begin{definition}[Framework, faithful UDG]
A \emph{realization} of $G=(V,E)$ is a map $p:V\to\R^2$; $(G,p)$ is a
\emph{framework}. It is a \emph{unit-distance realization} if $\norm{p_u-p_v}=1$
on edges, and \emph{faithful} if $\norm{p_u-p_v}\neq1$ on non-edges. A non-edge
at distance $1$ is a \emph{phantom edge}.
\end{definition}

\begin{definition}[Rigidity matrix, self-stress]
The \emph{rigidity matrix} $R(p)\in\R^{|E|\times2n}$ has, in the row of
$\{u,v\}$, the vector $(p_u-p_v)$ in the columns of $u$ and $(p_v-p_u)$ in $v$.
The framework is infinitesimally rigid iff $\rank R(p)=2n-3$. The right null
space consists of infinitesimal motions ($3$ isometries $+$ mechanisms); the left
null space consists of \emph{self-stresses}. A rigid framework is
\emph{isostatic} if $|E|=2n-3$ (no self-stress) and \emph{over-braced} if
$|E|>2n-3$.
\end{definition}

\begin{definition}[Laman graph and Laman number]\label{def:laman}
By Laman's Theorem~\cite{Laman}, $G$ on $n$ vertices is minimally rigid
(generically isostatic) in the plane iff $|E|=2n-3$ and $|E(H)|\le2|V(H)|-3$ for
every subgraph $H$; such $G$ is a \emph{Laman graph}. For generic squared edge
lengths $\lambda\in\C^E$, with a base edge anchored, the system
$\norm{p_i-p_j}^2=\lambda_{ij}$ is $0$-dimensional; its number of complex
solutions is finite and independent of the generic $\lambda$. This is the
\emph{Laman number} $c(G)$, the number of complex realizations up to isometry;
$c=2$ for the triangle, $c=24$ for the triangular prism. It satisfies a tropical
recursion~\cite{Capco}.
\end{definition}

\begin{definition}[Global rigidity; degree; Galois group]
$(G,p)$ is \emph{globally rigid} if every realization with the same edge lengths
is congruent to $p$. Hendrickson's necessary conditions for generic global
rigidity in the plane are $3$-connectivity and redundant rigidity~\cite{Hendrickson}.
For a field extension $K/\Q$, $[K:\Q]=\dim_\Q K$ and the tower law
$[K:\Q]=[K:L][L:\Q]$ holds; $\Gal(K/\Q)$ is the group of $\Q$-automorphisms. A
number is straightedge-and-compass constructible iff it lies in a tower of
quadratic extensions; origami additionally realizes cubics.
\end{definition}

\section{Rigidity and the complete heptagon}

\begin{proposition}\label{prop:rigidity}
On realizations certified to $e_{\max}\sim10^{-16}$: $\{35,1\}$ has
$\rank R=35=2n-3$, $0$ self-stresses (isostatic, not redundantly rigid, hence
not globally rigid); $\{21,2\}$ has $\rank R=39=2n-3$ and $3$ self-stresses
(over-braced, $D_g=-3$, redundantly rigid, hence globally rigid).
\end{proposition}

\begin{definition}[Huzita--Hatori folds $O5$, $O6$]
$O5$ takes a point $P$ onto a line $r$ through $Q$; it equals the intersection of
$r$ with the unit circle about $Q$, solving a quadratic. $O6$ (Beloch) takes
$P_1\to r_1$ and $P_2\to r_2$ simultaneously, solving a cubic.
\end{definition}

For completeness we record the full set of Huzita--Hatori fold axioms and mark the
two this construction relies on (Table~\ref{tab:folds}, Figure~\ref{fig:folds}).
Only $O6$ adds cube roots; $O1$--$O5$ and $O7$ are exactly straightedge-and-compass.
The complete $\{21,2\}$ provably needs the cubic $O6$. The braced $\{35,1\}$ places
its $14$ inner vertices by the quadratic fold~\eqref{eq:O5}, but its three free
angles are not free: the closure forces the heptagon angles $\tfrac{5\pi}{7},
\tfrac{2\pi}{7}$ and a third angle whose cosine is quadratic over
$\Q(\cos\tfrac{2\pi}{7})$ (Theorem~\ref{thm:origami35}). It therefore needs one
cubic $O6$ (trisection) and then quadratics --- origami, not straightedge-and-compass,
but origami all the same. Its generic realization count carries the primes $103,587$
(Section~\ref{sec:origami-numbers}); these belong to the generic complex variety, not
to the folded unit shape.

\begin{table}[h]\centering\small
\caption{The seven Huzita--Hatori fold axioms and the two used here. Only $O5$
(a quadratic) and $O6$ (the cubic Beloch fold) carry algebraic weight in this
paper.}
\label{tab:folds}
\begin{tabular}{@{}c l l l@{}}
\toprule
fold & geometric operation & power & used\\\midrule
$O1$ & line through two points $P,Q$                        & incidence          & implicitly\\
$O2$ & fold $P$ onto $Q$ (perpendicular bisector)           & linear             & implicitly\\
$O3$ & fold line $r$ onto line $s$ (bisector)               & linear             & --\\
$O4$ & crease through $P$ perpendicular to $r$              & linear             & --\\
$O5$ & fold $P$ onto line $r$, crease through $Q$           & quadratic $\sqrt{\ }$ & \textbf{yes}: braces of $\{35,1\}$\\
$O6$ & fold $P_1\!\to r_1$, $P_2\!\to r_2$ at once (Beloch) & cubic $\sqrt[3]{\ }$  & \textbf{yes}: complete $\{21,2\}$\\
$O7$ & fold $P$ onto $r_1$, crease perpendicular to $r_2$   & quadratic          & --\\
\bottomrule
\end{tabular}
\end{table}

\begin{figure}[t]\centering
\includegraphics[width=.9\textwidth]{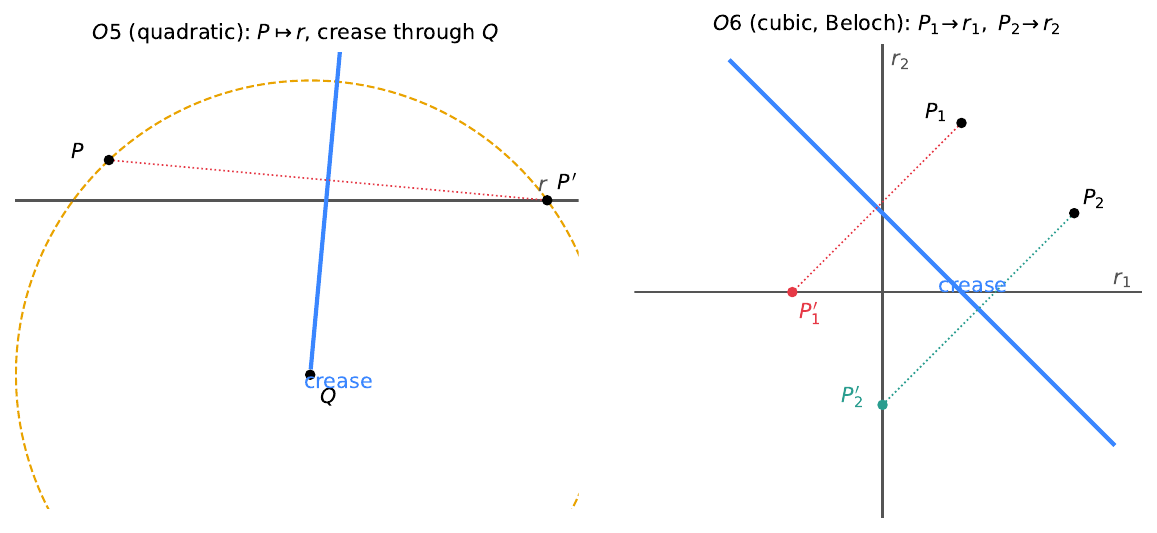}
\caption{The two folds used. \emph{Left,} $O5$ (quadratic): the crease reflects a
point $P$ onto a line $r$ while passing through a fixed point $Q$ --- equivalently
$r$ meets the circle about $Q$ through $P$. \emph{Right,} $O6$ (the cubic Beloch
fold): a single crease reflects $P_1$ onto $r_1$ and $P_2$ onto $r_2$
simultaneously; the crease is a common tangent of two parabolas, and its slope
solves a cubic.}
\label{fig:folds}
\end{figure}

\begin{theorem}[Closed form, geometric certificate, and field of $\{21,2\}$]\label{thm:complete}
The unique realization of $\{21,2\}$ consists of three concentric unit-edge
star-heptagons $\{7/1\},\{7/2\},\{7/3\}$ (Figure~\ref{fig:complete}): for layer
$r\in\{1,2,3\}$ and $k\in\{0,\dots,6\}$,
\[
 p_{r,k}=\rho_r\,(\cos\theta_{r,k},\sin\theta_{r,k}),\qquad
 \rho_r=\frac{1}{2\sin(r\pi/7)},\qquad
 \theta_{r,k}=\frac{2\pi k}{7}+\delta_r,
\]
with layer offsets $(\delta_1,\delta_2,\delta_3)=\bigl(0,\tfrac{\pi}{21},-\tfrac{\pi}{21}\bigr)$.
All $42$ edges then have length \emph{exactly} $1$, certified symbolically: each layer
edge is a step-$r$ chord, $2\rho_r\sin(r\pi/7)=1$; and each brace closes by the law of
cosines --- the $\{1,2\}$ and $\{1,3\}$ braces subtend $60^\circ$ with
$\rho_r^2+\rho_s^2-\rho_r\rho_s=1$, and the $\{2,3\}$ braces subtend $120^\circ$ with
$\rho_2^2+\rho_3^2+\rho_2\rho_3=1$. Hence the realization is \emph{exact}: all coordinates lie in $\Q(\cos\tfrac{\pi}{42})$,
the maximal real subfield of $\Q(\zeta_{84})$, of degree $12$, which contains the cubic
$\Q(\cos\tfrac{2\pi}{7})$, and this degree is minimal over all placements
(Remark~\ref{rem:degree12}); for instance $d(1,4)^2=1+4\cos\tfrac{2\pi}{7}+4\cos^2\tfrac{2\pi}{7}$.
Therefore $\{21,2\}$ requires the cubic $O6$ and is not straightedge-and-compass
constructible.
\end{theorem}

\begin{remark}[The degree $12$ is forced; $\{21,2\}$ is Haugland's $H_{21}$]\label{rem:degree12}
Every coordinate in Theorem~\ref{thm:complete} is a product of some $\rho_r$ with the
cosine or sine of a multiple of $\pi/21$, hence lies in $\Q(\cos\tfrac{\pi}{42})$; since
$\rho_2\cos\tfrac{\pi}{21}$ already has degree $12$, this is the coordinate field. No
congruent placement does better. Two edges of $\{21,2\}$ have directions differing by
$16\pi/21$, so the ratio of the two unit edge vectors is the primitive $21$st root of
unity $\zeta_{21}^{8}$, a quantity invariant under isometries. If the coordinates of
some placement lie in a real field $K$, then $K(i)\supseteq\Q(i,\zeta_{21})=\Q(\zeta_{84})$,
of degree $24$, whence $[K:\Q]\ge12$. The degree $12=2^2\cdot3$ is a $\{2,3\}$-number,
as Section~\ref{sec:origami-numbers} requires.

The graph $\{21,2\}$ --- a regular heptagon and two regular heptagrams joined by unit
braces --- is the graph $H_{21}$ that Haugland~\cite{Haugland} used as a source of edge
directions for a Moser-spindle-free $5$-chromatic unit-distance graph. Its two heptagram
layers $\{7/2\}\cup\{7/3\}$ are exactly the $14$-vertex seed from which Heule and
Przybocki~\cite{HeulePrzybocki} grow a $449$-vertex $5$-chromatic unit-distance graph;
their placement is the mirror image of ours rotated by $\pi/2$, and has its points in
$\Q(e^{2\pi i/21})$. Remark~\ref{rem:superrigid} shows that this direction set comes from
a single shape: the realization of $H_{21}$ is unique up to congruence.
\end{remark}

\begin{figure}[t]
\centering
\includegraphics[width=.60\textwidth]{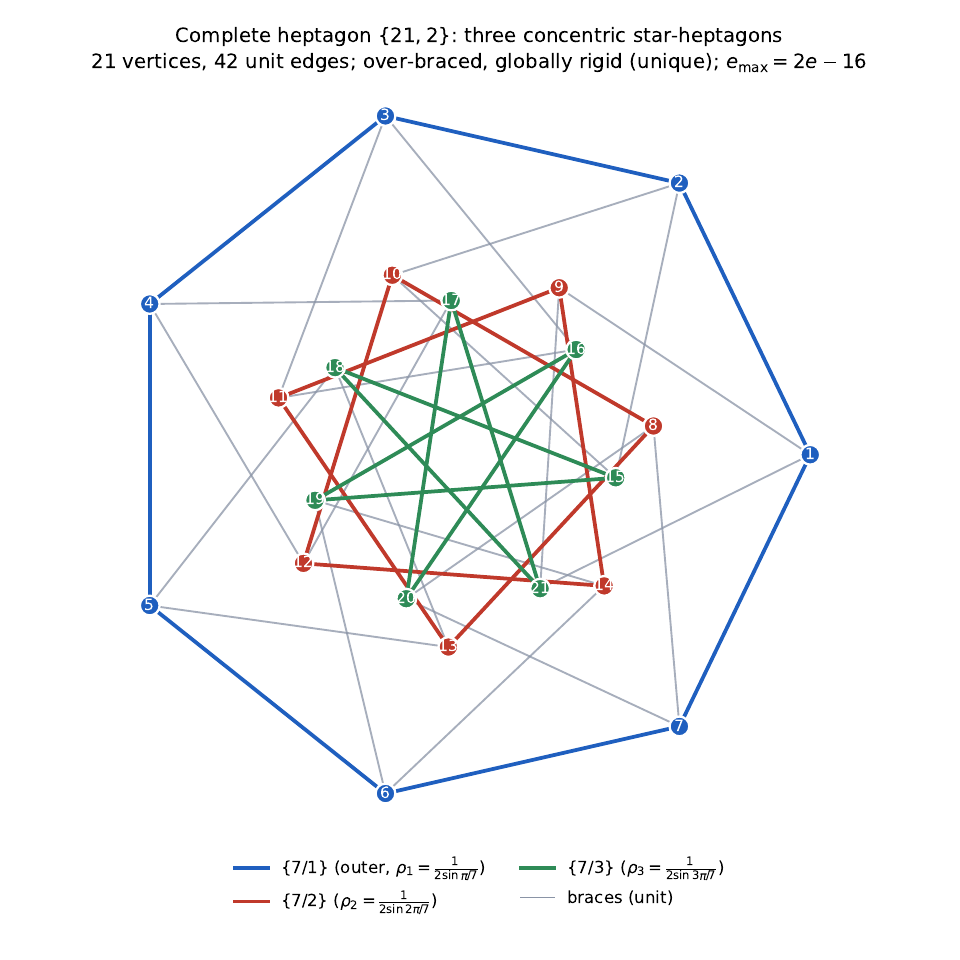}
\caption{The complete heptagon $\{21,2\}$: three concentric unit-edge star-heptagons
$\{7/1\}$ (blue), $\{7/2\}$ (red), $\{7/3\}$ (green), radii $\rho_r=1/(2\sin r\pi/7)$
and offsets $\delta_r\in\{0,\pm\pi/21\}$, joined by unit braces (grey). Over-braced
($42$ edges), globally rigid: the realization is unique. All edges certified unit,
$e_{\max}=2\times10^{-16}$ numerically and exactly by Theorem~\ref{thm:complete}.}
\label{fig:complete}
\end{figure}

\begin{remark}[The heptagon cubic versus the Delian cubic]\label{rem:cubics}
The cubic subfield $\Q(\cos\tfrac{2\pi}{7})$ has minimal polynomial
$8x^3+4x^2-4x-1$, whose discriminant $3136=56^2$ is a perfect square: its Galois
group is the \emph{cyclic} $C_3$, so it is a real, abelian (cyclotomic) cubic. It is
origami-constructible --- the Beloch fold $O6$ is necessary --- exactly as for the
classical Delian number $\sqrt[3]{2}$ of doubling the cube. The two cubics lie in the
\emph{same} origami tier yet are \emph{different} fields: $\sqrt[3]{2}$ has minimal
polynomial $x^3-2$, discriminant $-108$ (not a square), and non-abelian Galois group
$S_3$, and $\Q(\cos\tfrac{2\pi}{7})\not\ni\sqrt[3]{2}$. Thus the complete heptagon
needs a cube root of the \emph{cyclic} kind, not the Delian $\sqrt[3]{2}$.
\end{remark}

\begin{remark}[$\{21,2\}$ is globally rigid (super-rigid)]\label{rem:superrigid}
The over-bracing of $\{21,2\}$ makes it not merely rigid but \emph{globally} rigid.
Computationally, its rigidity matrix has generic rank $39=2n-3$ with $3$ self-stresses,
and it is \emph{redundantly rigid} (deleting any one of the $42$ edges preserves
$\rank R=2n-3$) and $3$-connected; by the Jackson--Jordán theorem~\cite{JacksonJordan}
it is therefore generically globally rigid, so its unit-distance realization is unique
up to congruence. This is the strong rigidity that the braced $\{35,1\}$ --- isostatic
but not redundantly rigid, hence with thousands of non-congruent realizations --- lacks.
\end{remark}

\begin{remark}[$\{35,1\}$ is $\{21,2\}$ minus two vertices]\label{rem:deletion}
The two heptagons are not merely analogous: the braced $\{35,1\}$ is obtained from the
complete $\{21,2\}$ by \emph{deleting two adjacent vertices} (both of degree $4$)
together with their incident edges --- removing $2\cdot4-1=7$ edges and passing from
$(21,42)$ to $(19,35)$; the resulting induced subgraph is isomorphic to $\{35,1\}$
(verified directly). Thus a single local deletion turns the over-braced, globally rigid
$\{21,2\}$ (degree-$12$ field over the heptagon cubic, super-rigid) into the isostatic $\{35,1\}$ (an
origami field of degree $3\cdot2^k$, thousands of non-congruent realizations; generic
count $2^{6}\cdot103\cdot587$). The complete heptagon enters here as the
\emph{antagonist} in rigidity of the braced one: they sit on opposite sides of the
rigidity divide, share the origami side of the constructibility divide, yet one is a
two-vertex deletion of the other.
\end{remark}

\section{Solving the heptagon cubic by folding}
\label{sec:lill}
The cubic at the core of the field of $\{21,2\}$ is $\Q(\cos\tfrac{2\pi}{7})$, with minimal
polynomial $8x^3+4x^2-4x-1$; by Remark~\ref{rem:cubics} it is origami- but not
straightedge-and-compass constructible. It is worth seeing \emph{how} paper folding
produces this root, since that is the mechanism behind every vertex of the
construction. Two classical tools combine: Lill's method~\cite{Lill}, which reads
the real roots of a polynomial off a right-angled path, and Beloch's
fold~\cite{Beloch,Alperin}, the single fold that performs the cubic step.

\paragraph{Lill's path.} Lay the coefficients $8,4,-4,-1$ as a right-angled path
from $O$, turning $90^\circ$ after each segment (Figure~\ref{fig:lill}, left) and
ending at $T$. A ray launched from $O$ that turns $90^\circ$ each time it meets the
next segment line and arrives at $T$ has launch slope equal to a real root: here
$-\tan\varphi=\cos\tfrac{2\pi}{7}=0.62349\ldots$, which we verify exactly, and the
three real roots $\cos\tfrac{2\pi k}{7}$ ($k=1,2,3$) are the three such rays.

\paragraph{Beloch's fold.} The ricochet has two interior right-angle turns, and a
single \emph{Beloch fold} --- the Huzita--Hatori axiom $O6$, folding two points onto
two lines simultaneously --- performs both at once (Figure~\ref{fig:lill}, right).
Equivalently $O6$ is a common tangent of two parabolas, whose slope solves a cubic:
exactly the cube-root/trisection power that straightedge and compass lack, and what
supplies the cubic layer of the coordinate field of $\{21,2\}$; the remaining steps
are quadratic. Multi-fold origami
extends this to arbitrary degree~\cite{AlperinLang}, but the heptagon needs only the
one cubic fold; and the braced $\{35,1\}$, as we show in Theorem~\ref{thm:origami35},
needs exactly the same cubic trisection followed by quadratics --- its unit
realization lives in the same origami world, despite a generic realization count that
carries the primes $103,587$.

\begin{figure}[t]\centering
\includegraphics[width=\linewidth]{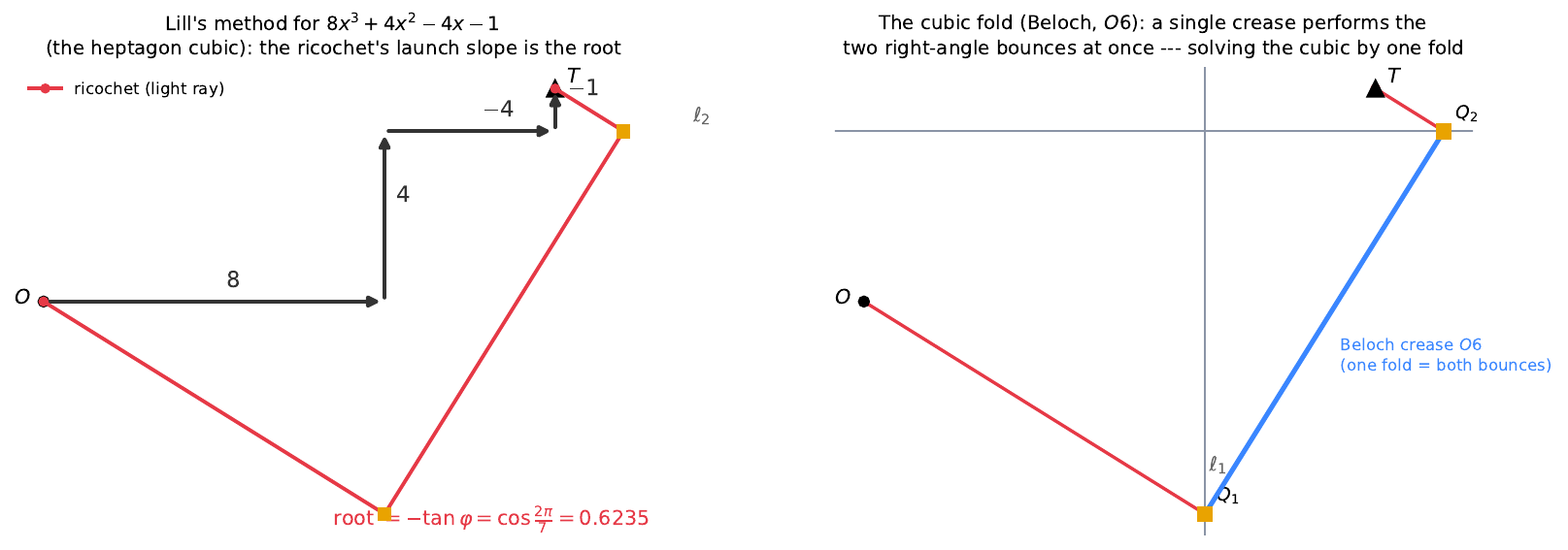}
\caption{Solving the heptagon cubic $8x^3+4x^2-4x-1$ by folding. \emph{Left:} Lill's
method --- the coefficient path $8,4,-4,-1$ (black) and the ricochet (red) whose
launch slope is the root $\cos\tfrac{2\pi}{7}=0.62349\ldots$ (verified exactly).
\emph{Right:} the two right-angle turns of the ricochet are realized by a single
Beloch fold ($O6$), the cubic-solving fold; its crease encodes the root.}
\label{fig:lill}
\end{figure}

The Beloch fold has a clean conic description that is worth a second picture: it is
a \emph{common tangent} of two parabolas --- one with focus $P_1$ and directrix
$r_1$, the other with focus $P_2$ and directrix $r_2$ --- since a line is tangent to
a parabola exactly when the reflection of the focus across the line lands on the
directrix. A pair of parabolas has up to \emph{three} real common tangents, the
three real folds, and their slopes are the three real roots of the cubic
(Figure~\ref{fig:parabolas}).

\begin{figure}[t]\centering
\includegraphics[width=.58\linewidth]{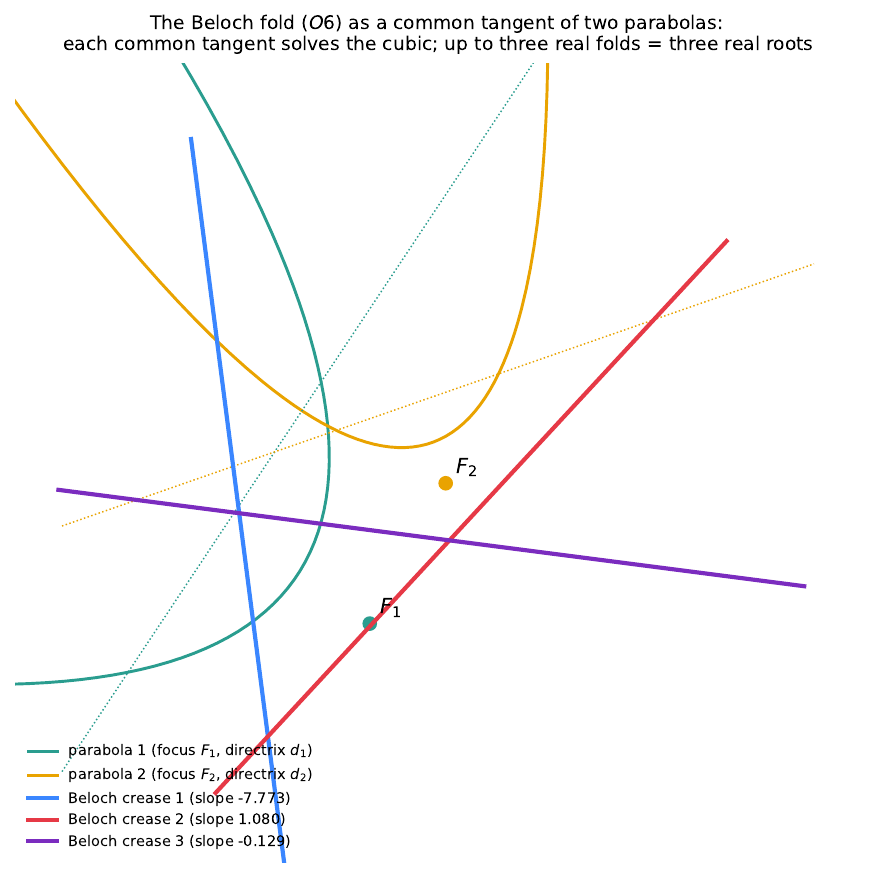}
\caption{The Beloch fold ($O6$) as a common tangent of two parabolas. A pair of
parabolas (foci $F_1,F_2$, directrices $d_1,d_2$) has up to three real common
tangents --- the three real Beloch folds --- whose slopes are the three real roots
of the cubic.}
\label{fig:parabolas}
\end{figure}

\section{The braced heptagon and the tower of fields}

The certifier builds a realization by continuity reduction: anchor a base edge;
place each triangle vertex by the circle intersection
\begin{equation}\label{eq:O5}
p_v=\tfrac{c_1+c_2}{2}\pm\sqrt{\tfrac{1}{\norm{c_2-c_1}^2}-\tfrac14}\,\bigl(-(c_2-c_1)^{\perp}\bigr)
\end{equation}
(fold $O5$, the sign being the orientation); free vertices carry parameters
$\theta$; the remaining edges give closure conditions $g_k(\theta)=0$. Each
realization is certified by Gram--SVD (planarity, edge residual, injectivity,
faithfulness). Since $\{35,1\}$ is a Laman graph, its Laman number (computed by
the recursion of~\cite{Capco}) is
\begin{equation}\label{eq:N}
N=c(\{35,1\})=3\,869\,504=2^{6}\cdot 103\cdot 587.
\end{equation}
Our fold search, refined by Newton to machine precision and deduplicated by
Procrustes, certified $\ge585$ real faithful realizations ($\approx0.015\%$ of $N$);
Figure~\ref{fig:braced-real} shows twenty-four of them.

\begin{figure}[t]
\centering
\includegraphics[width=\textwidth,trim=0 0 0 48,clip]{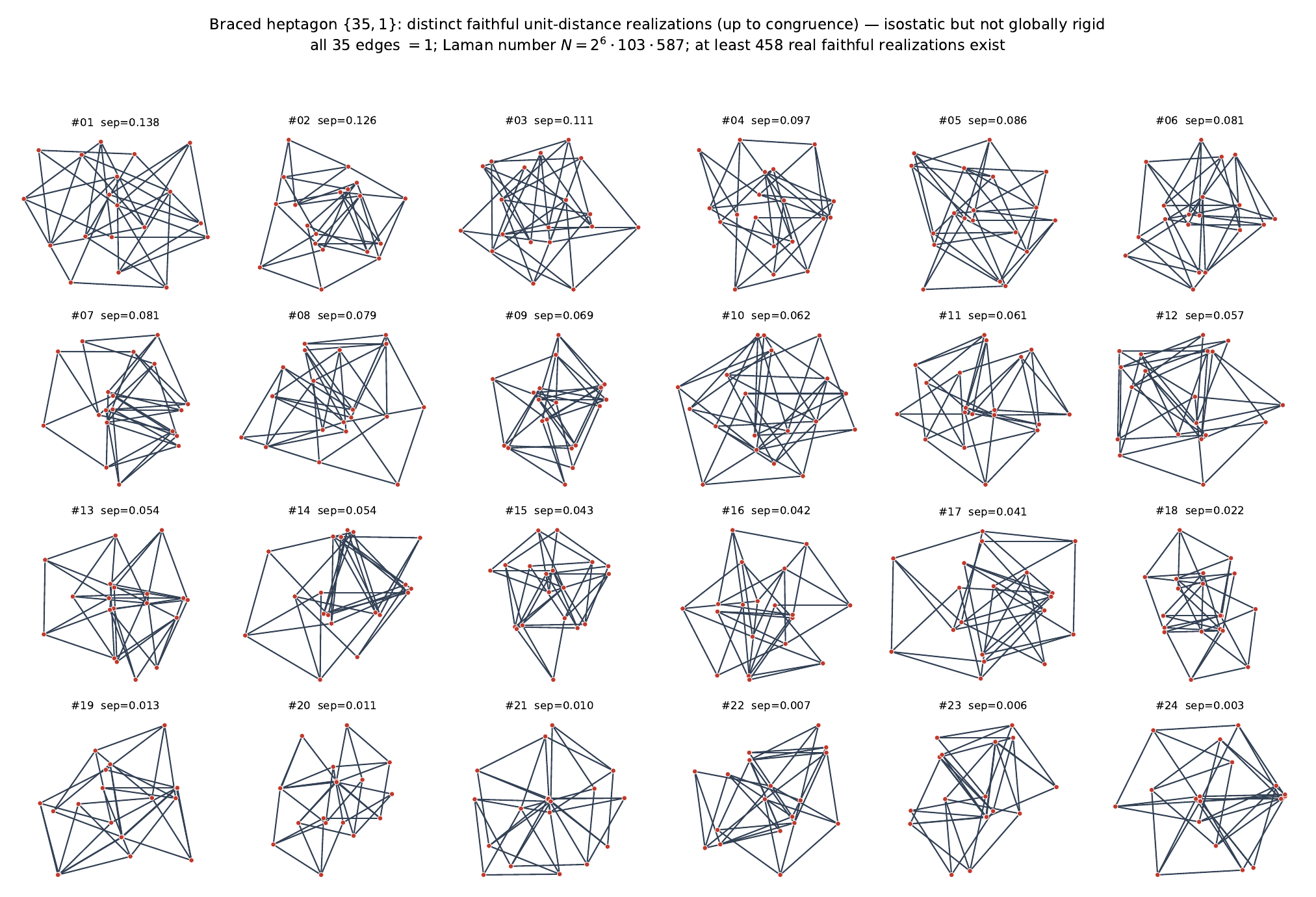}
\caption{Twenty-four distinct faithful unit-distance realizations of the braced
heptagon $\{35,1\}$, up to congruence (each panel labelled by index and minimum vertex
separation). All $35$ edges are unit in every panel; the graph is isostatic yet not
globally rigid, so its realization is far from unique --- a visual face of the Laman
number $N=2^{6}\cdot103\cdot587$. Compare the single, rigid realization of the complete
$\{21,2\}$ in Figure~\ref{fig:complete}.}
\label{fig:braced-real}
\end{figure}

\begin{theorem}[Tower of fields]\label{thm:tower}
Let $p$ be a faithful realization of $\{35,1\}$ with base anchored, $K=\Q(p)$.
There is a tower $\Q\subseteq L\subseteq K$ with $L=\Q(\theta_1,\theta_2,\theta_3)$,
$[L:\Q]=d$, and $K=L(\sqrt{f_1},\dots,\sqrt{f_s})$, $f_v\in L$, so $[K:L]=2^s$.
Hence $[K:\Q]=2^s d$, and the odd part of $[K:\Q]$ divides $d$: the field is
governed by the closure system in the free angles, not by the $O5$ folds
(Theorem~\ref{thm:origami35} evaluates $d=6$ for the unit realization).
\end{theorem}

\begin{proof}
By~\eqref{eq:O5} each triangle vertex adjoins $\sqrt{f_v}$ with $f_v$ in the field
generated so far, so $K/L$ is $2$-primary. The parameters satisfy $g_k(\theta)=0$;
clearing the radicals (introducing branch signs) yields a polynomial system whose
solution has degree $d$. Multiplicativity gives $[K:\Q]=2^s d$, and every odd
prime of $[K:\Q]$ divides $d$.
\end{proof}

The step from the \emph{count} $N$ to the coordinate field of a \emph{unit}
realization is where care is needed: $N$ counts the complex realizations of
$\{35,1\}$ under \emph{generic} edge lengths, whereas the unit realization is a
highly non-generic specialization. Pegg's explicit construction pins that
specialization down exactly. Anchoring $A=0$, $B=1$ and writing the three free
parameters as $a_1,a_2,a_3$, the braced regular heptagon is the solution
\begin{equation}\label{eq:pegg12}
\begin{gathered}
a_1=\tfrac{5\pi}{7},\qquad a_3=\tfrac{2\pi}{7},\qquad a_2=\arg\zeta,\quad\text{where}\\
7\zeta^{12}+28\zeta^{11}+70\zeta^{10}+133\zeta^{9}+203\zeta^{8}+259\zeta^{7}+281\zeta^{6}\\
{}+259\zeta^{5}+203\zeta^{4}+133\zeta^{3}+70\zeta^{2}+28\zeta+7=0,
\end{gathered}
\end{equation}
at which the three unused edges close exactly, $d(11,12)=d(13,14)=d(15,2)=1$. The
degree-$12$ polynomial in \eqref{eq:pegg12} is irreducible and palindromic; its
Chebyshev reduction $w=\zeta+\zeta^{-1}=2\cos a_2$ satisfies the irreducible sextic
\begin{equation}\label{eq:sextic}
S(w)=7w^{6}+28w^{5}+28w^{4}-7w^{3}-14w^{2}+1 .
\end{equation}

\begin{theorem}[The braced heptagon is origami-constructible]\label{thm:origami35}
Over the cyclic cubic $F=\Q(\cos\tfrac{2\pi}{7})$ the sextic \eqref{eq:sextic}
factors into three (Galois-conjugate) quadratics; hence $2\cos a_2$ has degree $2$
over $F$, and the angle field $L=\Q(\cos\tfrac{2\pi}{7},\cos a_2)$ is obtained from
$\Q$ by the cyclic-cubic trisection followed by a single quadratic,
$[L:\Q]=6=2\cdot3$. With Theorem~\ref{thm:tower} the coordinate field of the unit
realization is then a $\{2,3\}$-tower, $[K:\Q]=3\cdot2^{k}$; every step is realized
by an $O5$ or $O6$ fold, so $\{35,1\}$ is \emph{origami-constructible}.
\end{theorem}

\begin{proof}
Adjoining a root of $8c^{3}+4c^{2}-4c-1$ (the minimal polynomial of
$\cos\tfrac{2\pi}{7}$) to $\Q$ splits $S$ into three quadratic factors, verified in
exact arithmetic; the factors are permuted by $\mathrm{Gal}(F/\Q)=C_3$. Thus
$2\cos a_2$ generates a quadratic extension of $F$ and $[L:\Q]=6$. By
Theorem~\ref{thm:tower}, $K=L(\sqrt{f_1},\dots,\sqrt{f_s})$ is $2$-primary over $L$,
so $[K:\Q]=6\cdot2^{s}=3\cdot2^{k}$. A field reached by a tower of quadratic and
cubic steps is exactly the field of origami numbers~\cite{Alperin}: $O6$ performs
the single cubic (trisection) step and $O5$ each quadratic.
\end{proof}

\begin{remark}[The primes $103,587$ are the generic count, not the field]\label{rem:count-not-field}
Theorem~\ref{thm:origami35} corrects the natural but wrong inference that the
coordinate field has degree $N$. The number $N=2^{6}\cdot103\cdot587$ counts the
complex realizations of $\{35,1\}$ under \emph{generic} edge lengths, and the
monodromy of Remark~\ref{rem:monodromy} shows that generic variety is irreducible.
The \emph{unit} realization is a non-generic specialization whose arithmetic is
governed by \eqref{eq:sextic}: an origami field of degree $3\cdot2^{k}$ in which
$103$ and $587$ play no role. This puts $\{35,1\}$ on the same footing as the
triangle-free Exoo--Ismailescu graphs $EI_{17}$ and $EI_{19}$, whose closure
polynomials (degrees $20$ and $12$) are likewise fully solvable step by step: the
braced heptagon is certified exactly by the same geometry, not left to numerics.
\end{remark}

\begin{remark}[Monodromy evidence for irreducibility; a Type-2 graph beyond the
known Galois theorem]\label{rem:monodromy}
The irreducibility of the \emph{generic} realization variety of $\{35,1\}$ (a
question independent of the unit specialization of Theorem~\ref{thm:origami35}) is not
settled by the available Galois-group theorem for minimally rigid graphs: Makhul,
Schicho and
Warren~\cite{MakhulSchichoWarren} compute the Galois group of every \emph{Type-1}
graph (built from a single edge by pure Henneberg-1 vertex additions), but
$\{35,1\}$ has minimum degree $3$ and no degree-$2$ vertex, so it admits \emph{no}
Henneberg-1 reduction: it is a \emph{Type-2} graph, outside their reach, and its
irreducibility is a priori open (cf.~\cite{DistanceComponents}). We test it
numerically. Over generic complex edge lengths we track a single realization by
segment homotopy around random loops in the length parameters; each loop permutes
the realizations (monodromy). From one realization, $3638$ loops connect
$\ge2252$ distinct realizations, and the orbit grows \emph{linearly, with no
plateau} (Figure~\ref{fig:monodromy}). A reducible variety with the seed trapped in
a small component would have saturated at a proper divisor of $N$; the unbounded
linear growth is instead the signature of a transitive monodromy, i.e.\ an
irreducible variety. This is strong evidence --- not a proof (the sample reaches
$0.06\%$ of $N$) --- that the \emph{generic} realization variety is irreducible over
$\Q$. It concerns generic edge lengths only; consistently, the \emph{unit} heptagon
is a non-generic specialization whose real faithful realizations lie in the origami
field of Theorem~\ref{thm:origami35} (degree $3\cdot2^{k}$), so the primes
$103,587$ are a property of the generic realization variety, not of any single unit
realization.
\end{remark}

\begin{figure}[t]\centering
\includegraphics[width=.66\textwidth]{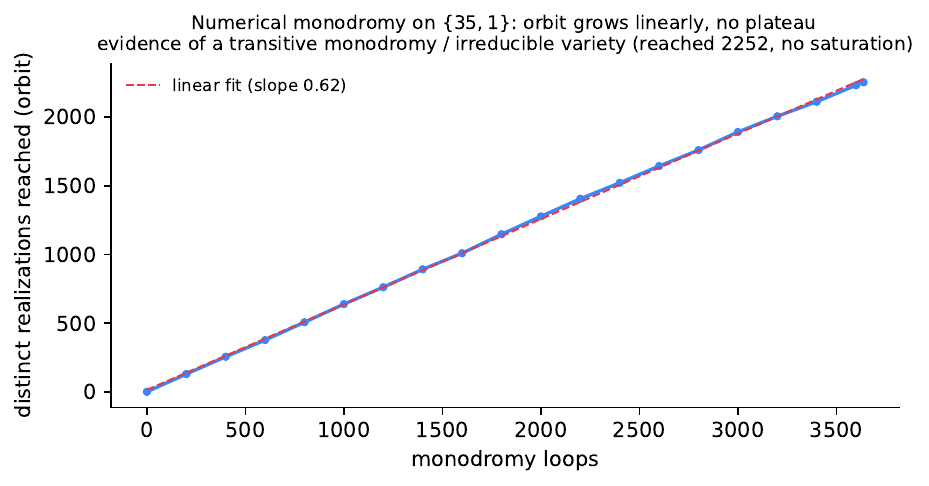}
\caption{Numerical monodromy on $\{35,1\}$ over generic complex lengths: the number
of distinct realizations reached from a single seed grows linearly in the number of
loops, with no saturation ($\ge2252$ reached). The absence of a plateau at a divisor
of $N=2^6\cdot103\cdot587$ is evidence of a transitive monodromy --- an irreducible
\emph{generic} realization variety --- whose unit specialization is the origami
field of Theorem~\ref{thm:origami35}.}
\label{fig:monodromy}
\end{figure}

\begin{table}[t]
\centering\small
\caption{The two heptagons: opposite in rigidity, kin in arithmetic. The generic
count $N$, the origami field of $\{35,1\}$, and the $\{21,2\}$ column are all
unconditional.}
\begin{tabular}{@{}lll@{}}
\toprule
 & $\{35,1\}$ (braced) & $\{21,2\}$ (complete)\\
\midrule
$n,\ |E|$          & $19,\ 35$                        & $21,\ 42$\\
rigidity           & isostatic ($D_g=0$)              & over-braced ($D_g=-3$)\\
complex realizations (generic) & $N=2^{6}\cdot103\cdot587$ & $1$ (unique)\\
construction       & origami: $O6$ then $O5$          & origami: $O6$ then $O5$\\
coordinate field   & $\{2,3\}$-tower, $[K:\Q]=3\cdot2^{k}$ & $\Q(\cos\tfrac{\pi}{42})$, $[K:\Q]=12$\\
\bottomrule
\end{tabular}
\end{table}

\section{The two heptagons and the field of origami numbers}
\label{sec:origami-numbers}
Alperin~\cite{Alperin} showed that the numbers constructible by Huzita--Hatori
folds form a field $\mathbb{O}$, the \emph{origami numbers}, lying strictly between
the ruler-and-compass field and the algebraic numbers and characterised by closure
under square \emph{and} cube roots: $\alpha\in\mathbb{O}$ iff $\alpha$ lies in a
tower $\Q=K_0\subset K_1\subset\cdots\subset K_r\ni\alpha$ with each
$[K_{i+1}:K_i]\in\{2,3\}$ --- equivalently, the Galois closure of $\Q(\alpha)$ is a
$\{2,3\}$-group. Both heptagons fall \emph{inside} this field, which is the arithmetic
kinship beneath their rigidity antagonism. For $\{21,2\}$ the coordinates lie in
$\Q(\cos\tfrac{\pi}{42})$, of degree $12$, reached by one cubic step (the cyclic
cubic, $C_3$) and two quadratics: \emph{every vertex is an origami number}, and the
Beloch fold of Section~\ref{sec:lill}, followed by $O5$ folds, constructs it. For $\{35,1\}$ the same picture
holds, one layer deeper. By Theorem~\ref{thm:origami35} the pinning angle satisfies
the sextic \eqref{eq:sextic}, which \emph{factors into quadratics} over the very cubic
$\Q(\cos\tfrac{2\pi}{7})$; so the coordinate field is a tower of one $C_3$ (trisection)
step and quadratics, of degree $3\cdot2^{k}$, and \emph{every vertex is again an
origami number}, built by the Beloch fold followed by $O5$ folds. The primes $103$ and
$587$ appear only in the generic complex count $N$, not in this tower --- a
$\{2,3\}$-tower yields exactly the degrees $2^a3^b$, and $3\cdot2^k$ is one of them.
Thus the over-braced, globally rigid heptagon and its isostatic two-vertex deletion
live on the \emph{same} side of $\mathbb{O}$; neither needs the multi-fold origami of
Alperin and Lang~\cite{AlperinLang}. Both sides are unconditional.

\section{Branch signs are mountain--valley assignments}
\label{sec:mv}

In the continuity reduction above, the base edge is anchored, three vertices carry
the free angles $\theta_1,\theta_2,\theta_3$, and the remaining $t=14$ vertices are
placed by the fold $O5$ of~\eqref{eq:O5}. Such a fold is the meeting of two unit
circles centred at already-placed neighbours $c_1,c_2$; formula~\eqref{eq:O5} exhibits
its \emph{two} solutions, mirror images of one another across the \emph{radical axis}
$\ell_v$ of the two circles --- the perpendicular bisector of $c_1c_2$ (equal radii),
which is the support of their common chord. Collect the $t$ orientation signs into a
vector $\varepsilon\in\{+,-\}^{t}$.

\begin{definition}[Crease and mountain--valley sign]\label{def:mv}
For the fold $O5$ placing $v$ from $c_1,c_2$, the radical axis $\ell_v$ is the
\emph{crease} of $v$, and the two solutions $p_v^{\pm}$ of~\eqref{eq:O5} are reflections
of each other across $\ell_v$. Declaring $\varepsilon_v=+$ a \emph{mountain} fold and
$\varepsilon_v=-$ a \emph{valley} fold, the triangle $c_1c_2v$ is laid to one side of
$\ell_v$ or to its mirror side. Thus $\varepsilon$ is a mountain--valley (M/V)
assignment to the $t$ creases.
\end{definition}

\begin{proposition}[Branch signs $=$ M/V calculus]\label{prop:mv}
Fix the free angles $\theta_1,\theta_2,\theta_3$. The map $\varepsilon\mapsto
p(\varepsilon)$ from sign vectors to configurations is the assignment of an M/V label to
each crease $\ell_v$; flipping a single $\varepsilon_v$ reflects the sub-framework beyond
$\ell_v$ across that crease (a local mountain$\leftrightarrow$valley reversal), producing
a \emph{non-congruent} framework, not an isometry of the whole. Hence the naive number
of configurations is $2^{t}=2^{14}$, cut down to those satisfying the closure system
$g_1=g_2=g_3=0$ --- the three unused edges forced to unit length.
\end{proposition}

\begin{proof}
Each triangle vertex has exactly the two pre-images $p_v^{\pm}$ of~\eqref{eq:O5},
interchanged by reflection in $\ell_v$; with the angles fixed, choosing $\varepsilon_v$
determines $p_v$ and, downstream, every vertex built from it, so $p(\varepsilon)$ is
well defined on $\{+,-\}^{t}$. A single reversal $\varepsilon_v\mapsto-\varepsilon_v$
replaces $p_v$ by its mirror across $\ell_v$ and reflects the vertices placed from it;
the resulting framework shares the $35$ edge lengths but has a different distance
multiset, hence is not congruent. The closure edges hold only on the zero set of the
$g_k$, which selects the consistent assignments.
\end{proof}

\begin{remark}[Closure $=$ foldability consistency]\label{rem:foldability}
The closure equations $g_k(\theta)=0$ play, for the UDG, the role that Maekawa's
condition ($\#M-\#V=\pm2$ at each interior vertex) and Kawasaki's condition
(alternating angle sum $=\pi$) play for a flat-foldable crease pattern~\cite{Hull}:
local binary fold choices, constrained by a global consistency that decides which
assignments actually close. A realization is \emph{faithful} exactly when the folded
configuration has no overlap producing a phantom unit distance --- the UDG analogue of
a valid, non-self-intersecting folded state.
\end{remark}

\begin{remark}[The realizable M/V patterns: an empirical enumeration]\label{rem:mv-empirical}
A direct search over the branch signs makes the sieve concrete. Sampling the closure
system and keeping only faithful states, we observe two facts. First, the branch-sign
vector $\varepsilon\in\{+,-\}^{14}$ \emph{determines} the realization: distinct faithful
realizations carry distinct M/V patterns, a bijection between realizations and
realizable patterns. Second, the realizable patterns occupy a tiny, structured part of
the $2^{14}=16384$ a priori choices. Writing $M=\#\{v:\varepsilon_v=+\}$ for the number
of mountains, the realizable patterns are confined to the band $M\in[2,13]$ ---
\emph{never} all-mountain or all-valley --- with a unimodal distribution peaked near
$M=7$ (Figure~\ref{fig:mv-hist}). This mountain--valley \emph{balance} is the
unit-distance shadow of Maekawa's $\#M-\#V=\pm2$: closure without a phantom fold forces a
near-even split of mountains and valleys. Figure~\ref{fig:mv-gallery} shows realizations
across the band, from $M=2$ to $M=13$. The $\ge585$ distinct faithful realizations
certified this way are still far below the complex total $N=2^6\cdot103\cdot587$, which
counts \emph{all} consistent M/V states over $\mathbb{C}$.
\end{remark}

\begin{figure}[t]\centering
\includegraphics[width=.72\textwidth]{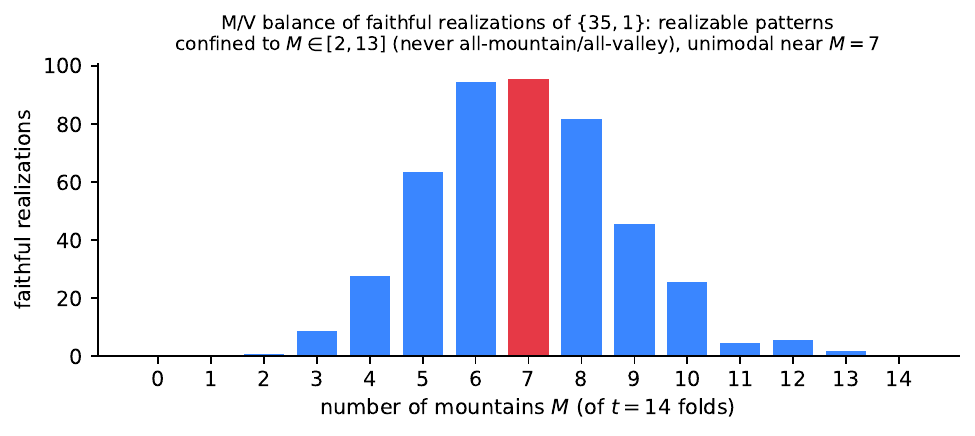}
\caption{Mountain--valley balance of the faithful realizations of $\{35,1\}$: the number
of mountains $M$ among the $t=14$ branch signs. Realizable patterns are confined to
$M\in[2,13]$ (never all-mountain or all-valley) and cluster near $M=7$ --- a
Maekawa-type balance forced by closure.}
\label{fig:mv-hist}
\end{figure}

\begin{figure}[t]\centering
\includegraphics[width=\textwidth]{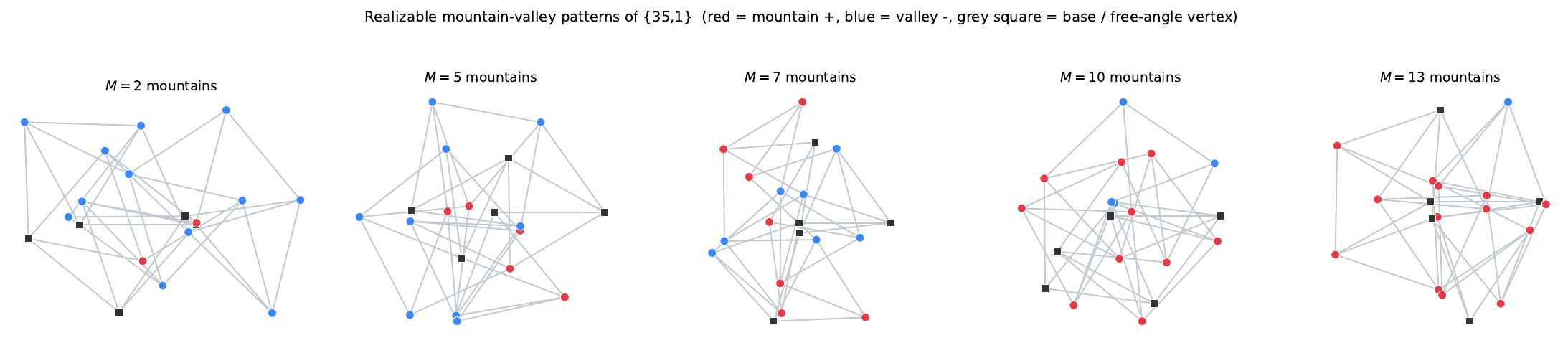}
\caption{Realizable M/V patterns across the band: faithful realizations of $\{35,1\}$
with $M=2,5,7,10,13$ mountains. Fold vertices are coloured red (mountain, $\varepsilon_v=+$)
or blue (valley, $\varepsilon_v=-$); grey squares mark the base edge and the three
free-angle vertices. Each realization is a distinct point of the closure variety, and
its M/V pattern is unique to it.}
\label{fig:mv-gallery}
\end{figure}

Complexifying and quotienting by isometry counts \emph{all} consistent M/V states over
generic lengths: this is the Laman number $N=2^{6}\cdot103\cdot587$ of~\eqref{eq:N}. The
real faithful realizations (the $\ge585$ certified) are the states with real coordinates
and no phantom edge. The Galois group $\Gal(K/\Q)$ of the coordinate field of the
\emph{unit} realization permutes these states; by Theorem~\ref{thm:origami35} that field
is the origami tower of degree $3\cdot2^k$, so $\Gal(K/\Q)$ is a $\{2,3\}$-solvable
group --- the primes $103,587$ live in the generic count $N$, not in the folded shape.
This is why the thousands of heptagon foldings, though far from unique, remain
\emph{ordinary} in the deepest sense: each is constructed by a trisection followed by
straightedge-and-compass folds, Galois conjugates inside Alperin's field of origami
numbers.

\begin{remark}[From correspondence to theorem]\label{rem:towards-theorem}
Proposition~\ref{prop:mv} and Remark~\ref{rem:foldability} give a \emph{structural}
correspondence --- a binary reflection at each crease, plus a global consistency ---
matching the mountain--valley calculus of origami. To promote it to a theorem one must
exhibit a crease pattern $\mathcal{C}$ whose rigidly foldable states are in natural
bijection with the faithful realizations of $\{35,1\}$, under which $\varepsilon$ is the
M/V map and the closure system $g_k=0$ coincides with the foldability (angle
consistency) equations of $\mathcal{C}$; one then reads the Galois action on
realizations as a symmetry of the M/V state space of $\mathcal{C}$. This is the precise
sense in which origami is not an illustration but the native construction model of this
family.
\end{remark}

\section{The question behind this work}

In correspondence between the authors~\cite{PeggLetters}, the second author asked whether the exact
solutions would use ``the same algebraic field as the regular heptagon'' (which
``would be a new form of rigidity''), or ``limited fields.'' The answer, by
Theorem~\ref{thm:origami35}, is the first: the braced $\{35,1\}$ \emph{does} use the
regular heptagon's own cubic field $\Q(\cos\tfrac{2\pi}{7})$ --- its pinning angle is
quadratic over exactly that cubic --- so its coordinate field is the origami tower
$\Q(\cos\tfrac{2\pi}{7})$-then-quadratics, of degree $3\cdot2^k$. The large count
$N=3\,869\,504$ is a property of the generic complex variety, not of the unit shape,
and the primes $103,587$ never enter the folded coordinates. The ``new form of
rigidity'' anticipated there is real and appears twice over: the globally rigid complete
$\{21,2\}$ carries the heptagon cubic (Theorem~\ref{thm:complete}), and the isostatic
braced $\{35,1\}$ carries it too, one quadratic layer up. That intuition was exactly
right: both heptagons speak the arithmetic of $\cos\tfrac{2\pi}{7}$.

\section*{Open problems}
\emph{(i)} Irreducibility of the \emph{generic} $\{35,1\}$ realization variety over
$\Q$ (Remark~\ref{rem:monodromy} gives numerical evidence; a proof remains open).
\emph{(ii)} The Galois group of the generic count and the geometric meaning of the
primes $103,587$. \emph{(iii)} The braced $\{7/k\}$
family: Laman numbers, fields, and $O5$ vs.\ $O6$. \emph{(iv)} The maximum number
of real faithful realizations. \emph{(v)} Formal (Lean/Coq) certification.
\emph{(vi)} The mountain--valley correspondence of Remark~\ref{rem:towards-theorem}:
a crease pattern whose rigidly foldable states biject with the faithful realizations of
$\{35,1\}$, with $g_k=0$ the foldability equations.

\section*{Code and data availability}
The Gram--SVD certifier (Python), the continuity-reduction search, and the C\raise.08ex\hbox{+}\kern-.05em\raise.08ex\hbox{+}\
Laman-number program (after~\cite{Capco}) are available from the first author; running
the latter on $\{35,1\}$ reproduces \eqref{eq:N} in about five minutes.

\section*{Acknowledgments}
This study was financed in part by the Coordenação de Aperfeiçoamento de
Pessoal de Nível Superior -- Brasil (CAPES) -- Finance Code 001.
The first author thanks his doctoral advisor, Gilson Alexandre Ostwald Pedro da
Costa, for his guidance and support, and Américo Barbosa da Cunha Junior for the
discussions on the geometry of origami.


\begin{thebibliography}{99}
\bibitem{Schaefer} M.~Schaefer, \emph{Realizability of graphs and linkages},
in: Thirty Essays on Geometric Graph Theory, Springer, 2013, pp.~461--482.
\bibitem{Laman} G.~Laman, \emph{On graphs and rigidity of plane skeletal
structures}, J.~Engrg.\ Math.\ \textbf{4} (1970) 331--340.
\bibitem{Capco} J.~Capco, M.~Gallet, G.~Grasegger, C.~Koutschan, N.~Lubbes,
J.~Schicho, \emph{The number of realizations of a Laman graph}, SIAM J.~Appl.\
Algebra Geom.\ \textbf{2} (2018) 94--125.
\bibitem{Hendrickson} B.~Hendrickson, \emph{Conditions for unique graph
realizations}, SIAM J.~Comput.\ \textbf{21} (1992) 65--84.
\bibitem{JacksonJordan} B.~Jackson, T.~Jord\'an, \emph{Connected rigidity matroids
and unique realizations of graphs}, J.~Combin.\ Theory Ser.\ B \textbf{94} (2005)
1--29.
\bibitem{Lill} E.~Lill, \emph{R\'esolution graphique des \'equations num\'eriques
de tout degr\'e}, Nouvelles Annales de Math\'ematiques (2)~\textbf{6} (1867) 359--362.
\bibitem{Beloch} M.~P.~Beloch, \emph{Sul metodo del ripiegamento della carta per la
risoluzione dei problemi geometrici}, Periodico di Mat.\ (4)~\textbf{16} (1936) 104--108.
\bibitem{Alperin} R.~C.~Alperin, \emph{A mathematical theory of origami
constructions and numbers}, New York J.\ Math.\ \textbf{6} (2000) 119--133.
\bibitem{AlperinLang} R.~C.~Alperin, R.~J.~Lang, \emph{One-, two-, and multi-fold
origami axioms}, in: Origami$^4$, A K Peters, 2009, pp.~371--393.
\bibitem{Hull} T.~Hull, \emph{Origametry: Mathematical Methods in Paper
Folding}, Cambridge Univ.\ Press, 2020.
\bibitem{MakhulSchichoWarren} M.~Makhul, J.~Schicho, A.~Warren, \emph{On Galois
groups of type-1 minimally rigid graphs}, Discrete Comput.\ Geom.\ (2025).
\bibitem{DistanceComponents} \emph{Irreducible components of sets of points in the
plane that satisfy distance conditions}, arXiv:2403.00392 (2024).
\bibitem{Haugland} J.~K.~Haugland, \emph{A Moser-spindle-free $5$-chromatic unit
distance graph on $2131$ vertices in the plane}, arXiv:2608.04542 (2026).
\bibitem{HeulePrzybocki} M.~J.~H.~Heule, B.~Przybocki, \emph{From two heptagrams to a
$449$-vertex $5$-chromatic unit-distance graph}, arXiv:2611.28518 (2026).
\bibitem{OpenAI6} OpenAI, \emph{The Euclidean plane is not five-colorable}, OpenAI Math
Release preprint, 2026.
\bibitem{PeggLetters} E.~Pegg~Jr., private correspondence between the authors
(letters), 2026.
\end{thebibliography}
\end{document}